\documentclass[reqno]{amsart}
\usepackage{amsmath}%
\usepackage{amsfonts}%
\usepackage{amssymb}%
\usepackage{graphicx}
\newtheorem{theorem}{Theorem}
\theoremstyle{plain}

\newtheorem{corollary}{Corollary}

\newtheorem{definition}{Definition}
\newtheorem{example}{Example}

\newtheorem{lemma}{Lemma}

\newtheorem{proposition}{Proposition}
\newtheorem{remark}{Remark}

\numberwithin{equation}{section}
\begin{document}
\begin{center}
\vspace*{1.3cm}

\textbf{FORMULAS FOR TRANSLATIVE FUNCTIONS}

\bigskip

by

\bigskip

PETRA WEIDNER\footnote{HAWK Hochschule f\"ur angewandte Wissenschaft und Kunst Hildesheim/\-Holzminden/ G\"ottingen University of Applied Sciences and Arts, Faculty of Natural Sciences and Technology,
D-37085 G\"ottingen, Germany, {petra.weidner@hawk.de}.}

\bigskip
\bigskip
Research Report \\ 
Version 2 from October 31, 2018\\
Improved Version of Version 1 from April 9, 2018
\end{center}

\bigskip
\bigskip

\noindent{\small {\textbf{Abstract:}}
In this report, we consider extended real-valued functions on some real vector space. Gerstewitz functionals are used to construct all translative functions. We derive formulas for translative functions which are lower semicontinuous, continuous or  have sublevel sets that are  given by linear inequalities. 
Each extended real-valued function is shown to be the restriction of some translative function to a hyperspace.
Continuity of an arbitrary extended real-valued function is characterized by its epigraph. Moreover, we study the directional closedness of sets as a base for the presented results.
}

\bigskip


\bigskip


\section{Introduction}

Translative functions are an important tool for scalarization in multicriteria optimization, decision theory, mathematical finance, production theory and operator theory since they can represent orders, preference relations and other binary relations \cite{Wei17b}.

They are the functions with uniform sublevel sets which can be generated by a linear shift of some set along a line. We will characterize the class of sets and vectors needed to construct all translative functionals on the one hand and all of these functionals which are lower semicontinuous, continuous or defined by linear inequalities on the other hand. This investigation is based on statements about the directional closedness of sets. 

The directional closedness of sets is investigated in Section \ref{s-direct}. Section \ref{s-def} contains the 
construction of translative functions and of the subclasses of semicontinuous and of continuous translative functions.
The case that the sublevel sets of a translative function are  given by linear inequalities is studied in Section \ref{sec-poly-fktal}.
In Section \ref{s-arbitfunc}, each extended real-valued function is described as the restriction of some translative functional, and interdependencies between the properties of both functions are proved.
Here, we will also characterize continuity by the epigraph of the function.

Throughout this paper, $Y$ is assumed to be a real vector space.
$\mathbb{R}$ and $\mathbb{N}$ will denote the sets of real numbers and of nonnegative integers, respectively.
We define $\mathbb{N}_{>}:=\mathbb{N}\setminus\{0\}$, $\mathbb{R}_{+}:=\{x\in\mathbb{R}\colon x\geq 0\}$, $\mathbb{R}_{>}:=\{x\in\mathbb{R}\colon x > 0\}$,
$\mathbb{R}_{+}^n:=\{(x_1,\ldots ,x_n)^T\in\mathbb{R}^n\colon x_i\geq 0 \mbox{ for all } i\in\{1,\ldots, n\}\}$ for each $n\in\mathbb{N}_>$.
Given any set $B\subseteq\mathbb{R}$ and some vector $k$ in $Y$, we will use the notation 
$B\; k:= \{b \cdot k\colon b \in B \}$.
A set $C\subseteq Y$ is a cone if $\lambda C\subseteq C$ for all $\lambda\in\mathbb{R}_{+}$. The cone $C$ is called nontrivial if $C\not=\emptyset$, $C\not=\{0\}$ and $C\not= Y$ hold. For a subset $A$ of $Y$, 
$\operatorname*{core}A$ will denote the algebraic interior (core) of $A$ and $0^+A:=\{u\in Y  \colon  A+\mathbb{R}_{+} u\subseteq A \}$ the recession cone of $A$. 
In a topological vector space $Y$, $\operatorname*{int}A$, $\operatorname*{cl}A$ and $\operatorname*{bd}A$ stand for the (topological) interior, the closure and the boundary, respectively, of $A$.
Consider a functional $\varphi :Y\to \overline{\mathbb{R}}$, where $\overline{\mathbb{R}}:=\mathbb{R}\cup \{-\infty ,+\infty\}$. Its effective domain is defined as 
$\operatorname*{dom}\varphi :=\{y\in Y \colon \varphi(y)\in\mathbb{R}\cup \{-\infty\}\}$. Its epigraph 
is $\operatorname*{epi}\varphi :=\{(y,t)\in Y\times \mathbb{R} \colon \varphi(y)\leq t\}$.
The sublevel sets of $\varphi $ are given as $\operatorname*{sublev}_{\varphi}(t):= \{y\in Y  : \varphi (y) \leq t\}$ with $t\in\mathbb{R}$.
$\varphi$ is said to be finite-valued on $F\subseteq Y$ if it attains only real values on $F$. It is said to be finite-valued if it attains only real values on $Y$. 
According to the rules of convex analysis, $\operatorname*{inf}\emptyset =+\infty .$

\section{Directional closedness of sets}\label{s-direct}

In this section, we introduce a directional closedness for sets. Since this property as well as recession cones are related to the construction of functions with uniform sublevel sets, we will investigate the directional closedness of sets and of their recession cones.

Throughout this section, $A$ will be a subset of $Y$ and $k\in Y\setminus\{0\}$.

We define the directional closure according to \cite{Wei17c}.

\begin{definition}\label{def-k-dir-clos}
The $k$-directional closure $\operatorname*{cl}_k(A)$ of $A$ consists of all elements $y\in Y$ with the following property:
For each $\lambda\in\mathbb{R}_>$, there exists some $t\in\mathbb{R}_+$ with $t<\lambda$ such that $y-tk\in A$.\\
$A$ is said to be $k$-directionally closed if $A=\operatorname*{cl}_k(A)$.
\end{definition}

\begin{remark}
Definition \ref{def-k-dir-clos} defines the $k$-directional closure in such a way that it is the closure into direction $k$.
In \cite{QiuHe13}, the $k$-vector closure of $A$ is defined as
$\operatorname*{vcl}_k(A):=\{ y\in Y\mid \exists (t_n)_{n\in\mathbb{N}} \mbox{ with } t_n\to 0 \mbox{ such that } t_n\in\mathbb{R}_+  \mbox{ and } y+t_nk\in A \mbox{ for all } n\in\mathbb{N}\} .$
Hence, $\operatorname*{vcl}_k(A)=\operatorname*{cl}_{-k}(A)$. This $k$-vector closure has been used for the investigation of Gerstewitz functionals under the assumption that $A$ is a convex cone and $k\in A\setminus (-A)$ in \cite{QiuHe13} and without this restrictive assumption in \cite{GutNovRodTan2016}. 
\end{remark}

Obviously, $\operatorname*{cl}_k(Y)=Y$ and $\operatorname*{cl}_k(\emptyset )=\emptyset$.

It is easy to verify that directional closedness can also be characterized in the following way \cite[Lemma 2.2]{Wei17c}.

\begin{lemma}\label{l-k-dir-clsd}
$A$ is $k$-directionally closed if and only if, for each $y\in Y$, we have
$y\in A$ whenever there exists some sequence $(t_n)_{n\in\mathbb{N}}$ of real numbers with $t_n\searrow 0$ and $y-t_nk\in A$.
\end{lemma}

Definition \ref{def-k-dir-clos} implies immediately \cite[Lemma 2.1]{Wei17c}:

\begin{lemma}\label{l-kcl}
\begin{itemize}
\item[]
\item[\rm (a)] $A\subseteq \operatorname*{cl}_k(A)$.
\item[\rm (b)] $\operatorname*{cl}_k(A)\subseteq \operatorname*{cl}_k(B)$ if $A\subseteq B\subseteq Y$.
\item[\rm (c)] $\operatorname*{cl}_k(A+y)=\operatorname*{cl}_k(A)+y$ for each $y\in Y$.
\end{itemize}
\end{lemma}

\begin{proposition}
$\operatorname*{cl}_k(\operatorname*{cl}_k A)= \operatorname*{cl}_k(A)$.
\end{proposition}
\begin{proof}
For $A=\emptyset$, the assertion is true. Assume now $A\not=\emptyset$.
\begin{itemize}
\item[\rm (i)] $A\subseteq \operatorname*{cl}_k(A)$ implies $ \operatorname*{cl}_k(A)\subseteq \operatorname*{cl}_k(\operatorname*{cl}_k A)$ because of Lemma \ref{l-kcl}.
\item[\rm (ii)] Take any $y\in \operatorname*{cl}_k(\operatorname*{cl}_k A)$ and $\lambda\in\mathbb{R}_>$.
By Definition \ref{def-k-dir-clos},  there exists some $t_0\in\mathbb{R}_+$ with $t_0<\frac{\lambda}{2}$ and $y-t_0k\in\operatorname*{cl}_k(A)$.
If $t_0=0$, then $y\in\operatorname*{cl}_k(A)$. Suppose now $t_0\in\mathbb{R}_>$.
By Definition \ref{def-k-dir-clos}, there exists some $t\in\mathbb{R}_+$ with $t<t_0$ and $(y-t_0k)-tk\in A$, i.e., $y-(t_0+t)k\in A$, where
$0<t+t_0<2t_0<\lambda $. Hence $y\in\operatorname*{cl}_k(A)$.
\end{itemize}
\end{proof}

Lemma \ref{l-k-dir-clsd} yields:

\begin{proposition}\label{kcl_posrecc}
$A$ is $k$-directionally closed if $k\in 0^+A$.
\end{proposition}

But in applications, we often have to deal with the case $k\in -0^+A$. Let us prove a statement from \cite[Lemma 2.3]{Wei17c}.

\begin{proposition}\label{kcl_recc}
If $k\in -0^+A$, then 
$\operatorname*{cl}_k(A)=\{y\in Y\colon y-\mathbb{R}_>k\subseteq A\}.$
\end{proposition}
\begin{proof}
\begin{itemize}
\item[]
\item[\rm (i)] Assume first $y\in\operatorname*{cl}_k(A)$. If $y\in A$, then $y-tk\in A$ holds for all $t\in\mathbb{R}_>$.
Suppose now $y\notin A$. Take any $t\in\mathbb{R}_>$. By Definition \ref{def-k-dir-clos}, there exists some $t_0\in\mathbb{R}_>$ with $t_0<t$ such that $y-t_0k\in A$. Then $y-tk=y-t_0k+(t_0-t)k\in A+0^+A\subseteq A$.
\item[\rm (ii)] Take any $y\in Y$ for which $y-tk\in A$ holds for all $t\in\mathbb{R}_>$. Then $y\in \operatorname*{cl}_k(A)$ by Definition \ref{def-k-dir-clos}.
\end{itemize}
\end{proof}

Recall the following definition.

\begin{definition}
$y\in Y$ is linearly accessible from $A$ if there exists some
$a\in A$ such that $a+\lambda (y-a)\in A$ for all real numbers $\lambda\in (0,1)$.
$A$ is said to be algebraically closed if it contains all elements of $Y$ which are linearly accessible from $A$. 
\end{definition}

In a topological vector space, each closed set is algebraically closed. The next two propositions connect algebraical and directional closedness \cite[Lemmas 2.1 and 5.2]{Wei17b}.

\begin{proposition}\label{l-dir-clsd-suff}
Assume that $A$ is algebraically closed and $k\in -0^+A\setminus\{0\}$.
Then $A$ is $k$-directionally closed.
\end{proposition}

\begin{proposition}\label{cone-alg-clsd}
Let $A$ be a nontrivial, convex cone with
$k\in-\operatorname*{core}A$. Then
$A$ is $k$-directionally closed if and only if $A$ is algebraically closed.
\end{proposition}

In Proposition \ref{cone-alg-clsd}, the assumption $k\in-\operatorname*{core}A$ cannot be replaced by the condition $k\in (-A)\setminus A$. In Example \ref{ex-lex}, $A$ is a nontrivial, convex cone which is $k$-directionally closed for some $k\in (-A)\setminus A$ though $A$ is not algebraically closed.

\begin{corollary}\label{cor-rec-clsd}
Assume $k\in-\operatorname*{core}0^+A$. Then
$0^+A$ is $k$-directionally closed if and only if $0^+A$ is algebraically closed.
\end{corollary}

\begin{lemma}
If $A$ is algebraically closed, then $0^+A$ is algebraically closed.
\end{lemma}
\begin{proof}
Take any $a^0\in 0^+A$ and $y\in Y$ for which $a^0+t(y-a^0)\in 0^+A$ holds for all $t\in (0,1)$.
Consider arbitrary elements $a\in A$ and $\lambda\in\mathbb{R}_>$.
Since $0^+A$ is a cone we get $\lambda a^0+t(\lambda y-\lambda a^0)\in 0^+A$ for all $t\in (0,1)$.
Thus $(a+\lambda a^0)+t((a+\lambda y)-(a+\lambda a^0))=a+\lambda a^0+t(\lambda y-\lambda a^0)\in A+0^+A\subseteq A$ for all $t\in (0,1)$.
Since $a+\lambda a^0\in A$ and $A$ is algebraically closed, we get $a+\lambda y\in A$.
Thus $y\in 0^+A$. Hence, $0^+A$ is algebraically closed.
\end{proof}

\begin{proposition}\label{p-varphi_rec}
If $A$ is $k$-directionally closed, then $0^+A$ is $k$-directionally closed.
\end{proposition}

This proposition is due to \cite[Prop. 2.4(c)]{Wei17b}. It yields, together with Corollary \ref{cor-rec-clsd}:

\begin{corollary}
Assume that $A$ is $k$-directionally closed and $k\in-\operatorname*{core}0^+A$. Then
$0^+A$ is algebraically closed.
\end{corollary}

Let us now investigate the relationship between directional and topological closedness in a topological vector space.

\begin{proposition}\label{p-kcl-cl}
Assume that $Y$ is a topological vector space.
\begin{itemize}
\item[{\rm (a)}] $\operatorname*{cl}A$ is $k$-directionally closed, i.e., $\operatorname*{cl}_k(\operatorname*{cl}A)=\operatorname*{cl}A$.
\item[{\rm (b)}] $\operatorname*{cl}_k A\subseteq \operatorname*{cl}A$.
\item[{\rm (c)}] If $A$ is closed, then it is $k$-directionally closed, i.e., $\operatorname*{cl}_k A=A$.
\item[{\rm (d)}] $\operatorname*{cl}_k A$ is closed if and only if $\operatorname*{cl}_k A=\operatorname*{cl}A$.
\item[{\rm (e)}] $\operatorname*{cl}A-\mathbb{R}_>k\subseteq A$ holds if and only if $k\in -0^+A$ and $\operatorname*{cl}_k A=\operatorname*{cl}A$.
\end{itemize}
\end{proposition}
\begin{proof}
\begin{itemize}
\item[]
\item[\rm (a)] Take any $y\in Y$ for which there exists some sequence $(t_n)_{n\in\mathbb{N}}$ of real numbers with $t_n\searrow 0$ and $y-t_nk\in \operatorname*{cl}A$. Since each neighborhood of $y$ contains some $y-t_nk$, we get $y\in \operatorname*{cl}(\operatorname*{cl}A)=\operatorname*{cl}A$. Hence, $\operatorname*{cl}A$ is $k$-directionally closed.
\item[{\rm (b)}] $A\subseteq \operatorname*{cl}A$ yields $\operatorname*{cl}_k A\subseteq \operatorname*{cl}_k(\operatorname*{cl}A)=\operatorname*{cl}A$ by (a).
\item[\rm (c)] results from (a).
\item[\rm (d)] Assume first that $\operatorname*{cl}_k A$ is closed. By Lemma \ref{l-kcl}, $A\subseteq \operatorname*{cl_k}A$. This yields $\operatorname*{cl}A\subseteq \operatorname*{cl}(\operatorname*{cl_k}A)=\operatorname*{cl_k}A$. This implies $\operatorname*{cl}A=\operatorname*{cl_k}A$ by (b). The reverse direction of the assertion is obvious.
\item[\rm (e)] $\operatorname*{cl}A-\mathbb{R}_>k\subseteq A$ implies $k\in -0^+A$ and $\operatorname*{cl}A\subseteq \operatorname*{cl}_k A$ by Proposition \ref{kcl_recc}; hence $\operatorname*{cl}_k A=\operatorname*{cl}A$ because of (b).\\
If $k\in -0^+A$ and $\operatorname*{cl}_k A=\operatorname*{cl}A$, then Proposition \ref{kcl_recc} yields $\operatorname*{cl}_k A-\mathbb{R}_>k\subseteq A$ and thus $\operatorname*{cl}A-\mathbb{R}_>k\subseteq A$.
\end{itemize}
\end{proof}

\begin{example}\label{ex-lex}
The lexicographic order in $\mathbb{R}^2$ is represented by the convex, pointed cone $A=\{ (y_1,y_2)^T\in\mathbb{R}^2 : y_1>0\}\cup \{ (y_1,y_2)^T\in\mathbb{R}^2 : y_1=0, y_2\geq 0\}$. 
$A$ is not algebraically closed and thus not closed.
Obviously, $0^+A=A$ and $-0^+A\setminus\{ 0\}=-A\setminus\{ 0\}=Y\setminus A=-A\setminus A$.
\begin{itemize}
\item[\rm (a)] For each $k\in A\setminus\{ 0\}$, $A$ is $k$-directionally closed by Proposition \ref{kcl_posrecc}.
\item[\rm (b)] For each $k\in -\operatorname*{core}A=\{(y_1,y_2)^T\in\mathbb{R}^2 : y_1<0\}$, we have $\operatorname*{cl}A-\mathbb{R}_>k\subseteq A$ and $\operatorname*{cl}_{k}A=\operatorname*{cl}A=\{ (y_1,y_2)^T\in\mathbb{R}^2 : y_1\geq 0\}$.
\item[\rm (c)] For each $k\in -A\setminus ((-\operatorname*{core}A)\cup\{ 0\})=\{ (y_1,y_2)^T\in\mathbb{R}^2 : y_1=0, y_2 < 0\}$, $A$ is $k$-directionally closed and $\operatorname*{cl}A-\mathbb{R}_>k\not\subseteq A$.
\end{itemize}
\end{example}

\section{Translative functions}\label{s-def}

If all sublevel sets of an extended real-valued function $\varphi : Y\to\overline{\mathbb{R}}$ have the same shape, then there exists some set $A\subseteq Y$ and some function $\zeta\colon \mathbb{R}\to Y$ with
\begin{equation}\label{eq-uniform}
\operatorname*{sublev}\nolimits_{\varphi }(t)=A+\zeta (t)\mbox{ for all }t\in\mathbb{R}.
\end{equation}
Clearly, $A$ and $\zeta$ have to fulfill the condition
$$ A+\zeta (t_1)\subseteq A+\zeta (t_2)\mbox{ for all }t_1,t_2\in\mathbb{R}\mbox{ with }t_1<t_2.$$
Condition (\ref{eq-uniform}) coincides with
\begin{equation*}
\operatorname*{epi}\varphi =\{ (y,t)\in Y\times\mathbb{R}\colon y\in A+\zeta (t)\}.
\end{equation*}
It implies
\begin{equation}\label{eq-uniconstruct}
\varphi (y)=\operatorname*{inf}\{ t\in\mathbb{R}\colon y\in A+\zeta (t)\}\mbox{ for all } y\in Y
\end{equation}
since 
$\varphi (y)=\operatorname*{inf}\{ t\in\mathbb{R}\colon \varphi (y)\leq t\}$.

If there exists some fixed vector $k\in Y\setminus\{ 0\}$ such that $\zeta (t)=tk$ for each $t\in\mathbb{R}$, then (\ref{eq-uniform}) is equivalent to
\begin{equation}\label{eq-lev0}
\operatorname*{sublev}\nolimits_{\varphi }(t)=\operatorname*{sublev}\nolimits_{\varphi }(0)+tk\mbox{ for all }t\in\mathbb{R},
\end{equation}
and (\ref{eq-uniconstruct}) coincides with the definition of a Gerstewitz functional.\\ 
A Gerstewitz functional
$\varphi _{A,k} : Y\to\overline{\mathbb{R}}$ is given by
\begin{equation*}
\varphi _{A,k}(y):=\operatorname*{inf}\{ t\in\mathbb{R}\colon y\in A+tk\}\mbox{ for all } y\in Y
\end{equation*}
with $A\subseteq Y$ and $k\in Y\setminus\{ 0\}$.\\
This formula was introduced by Gerstewitz (later Gerth, now Tammer) for convex sets $A$ under more restrictive assumptions in the context of vector optimization \cite{ger85}.
Basic properties of this function on topological vector spaces $Y$ have been proved in \cite{GerWei90} and \cite{Wei90}, later followed by \cite{GopRiaTamZal:03}, \cite{TamZal10}, \cite{DT} and \cite{Wei17a}. 
For detailed bibliographical notes, see \cite{Wei17a}. There it is also pointed out that researchers from different fields of mathematics and economic theory have applied Gerstewitz functionals. Properties of Gerstewitz functionals on linear spaces were studied in \cite{Wei17b} and \cite{Wei17c}.

\begin{example}
Assume that $Y$ is the vector space of functions $f\colon X\to\mathbb{R}$, where $X$ is some nonempty set.
Choose $k\in Y$ as the function with the constant value 1 at each $x\in X$ and
$A:=\{f\in Y\colon f(x)\leq 0\mbox{ for all }x\in X\}$. Then we get, for each $f\in Y$,
\begin{eqnarray*}
\varphi_{A,k}(f) & = & \operatorname*{inf}\{t\in\mathbb{R}\colon f\in A+tk\}\\
& = &  \operatorname*{inf}\{t\in\mathbb{R}\colon f-tk\in A\}\\
& = &  \operatorname*{inf}\{t\in\mathbb{R}\colon f(x)-t\leq 0\mbox{ for all }x\in X\}\\
& = &  \operatorname*{inf}\{t\in\mathbb{R}\colon f(x)\leq t\mbox{ for all }x\in X\}\\
& = &  \operatorname*{sup}_{x\in X}f(x),\mbox{ where }\\
\operatorname*{dom}\varphi_{A,k} & = & \{f\in Y\colon \operatorname*{sup}_{x\in X}f(x)\in \mathbb{R}\}.
\end{eqnarray*}
\end{example}

Each Gerstewitz functional fulfills equation (\ref{eq-lev0}).

\begin{proposition} \label{prop-vor-theo-and}
Consider a function $\varphi:Y\rightarrow \overline{{\mathbb{R}}}$ and $k\in Y\setminus\{ 0\}$.\\
The following conditions are equivalent to each other:
\begin{eqnarray} 
\operatorname*{sublev}\nolimits_{\varphi}(t) & = & \operatorname*{sublev}\nolimits_{\varphi}(0)+tk \quad \mbox{ for all } t \in\mathbb{R},\label{K4-sub}\\
\operatorname*{epi}\varphi & = & \{(y,t)\in Y\times \mathbb{R}  : y\in \operatorname*{sublev}\nolimits_{\varphi}(0)+tk\}, \\
\varphi (y+t k) & = & \varphi (y)+t \quad \mbox{ for all } y\in Y, \;t \in\mathbb{R},\label{trans_sub}\\
\varphi (y) & = & \inf \{t\in
{\mathbb{R}}  : y\in \operatorname*{sublev}\nolimits_{\varphi}(0)+tk \} \;\;\mbox{ for all } y\in Y.\label{gerst_sub}
\end{eqnarray} 
\end{proposition}

$\varphi$ is said to be $k$-translative if property (\ref{trans_sub}) is satisfied. It is said to be translative if it is $k$-translative for some $k\in Y\setminus\{ 0\}$.

\begin{remark}
Proposition \ref{prop-vor-theo-and} was proved in \cite[Proposition 3.1]{Wei17c}.
The $k$-translativity of $\varphi _{A,k}$ had already been mentioned in \cite{GopTamZal:00}.
Hamel \cite[Proposition 1]{Ham12} studied relationships between Gerstewitz functionals and $k$-translative functions, where he also introduced a notion of directional closedness. But his definition of $k$-directional closedness is different from the one we use. 
He defined $A$ to be $k$-directionally closed if, for each $y\in Y$, one has
$y\in A$ whenever there exists some sequence $(t_n)_{n\in\mathbb{N}}$ of real numbers with $t_n\rightarrow 0$ and $y\in A-t_nk$.
Hamel  proved the equivalence of (\ref{K4-sub})-(\ref{trans_sub}) \cite[Theorem 1]{Ham06} and that (\ref{trans_sub}) implies (\ref{gerst_sub}) \cite[Proposition 1]{Ham12}.
\end{remark}

Note that the sublevel sets of a Gerstewitz functional $\varphi_{A,k}$ do not always have the shape of $A$. But they can be described using $A$ and $k$ in the following way \cite[Theorem 3.1]{Wei17c}.

\begin{theorem}\label{t-sublevequ}
Assume $A\subseteq Y$ and $k\in Y\setminus\{ 0\}$.
Consider $\tilde{A}:=\operatorname*{sublev}\nolimits_{\varphi_{A,k}}(0)$.
\begin{itemize}
\item[\rm (a)] $\tilde{A}$ is the unique set for which  
\begin{equation*}
\operatorname*{sublev}\nolimits_{\varphi_{A,k}}(t)=\tilde{A}+tk \mbox{ for all } t \in\mathbb{R}
\end{equation*}
holds.
\item[\rm (b)] $\tilde{A}$ is the unique set with the following properties:
\begin{itemize}
\item[\rm (i)] $\quad\tilde{A}$ is $k$-directionally closed,
\item[\rm (ii)] $\quad k\in -0^+\tilde{A}\setminus\{0\}$ and
\item[\rm (iii)] $\quad\varphi _{A,k}$ coincides with $\varphi _{\tilde{A},k}$ on $Y$.
\end{itemize}
\item[\rm (c)] $\tilde{A}$ is the $k$-closure of $A-\mathbb{R}_{+}k$. It
consists of those points $y\in Y$ for which $y-tk\in A-\mathbb{R}_{+}k$ holds for each $t\in\mathbb{R}_>$.
\end{itemize}
\end{theorem} 

Proposition \ref{prop-vor-theo-and} and Theorem \ref{t-sublevequ} completely characterize the class of functions with uniform sublevel sets which are translative. This is summarized in the following theorem \cite[Theorem 3.2]{Wei17c}. 

\begin{theorem}\label{t-alle}
For each $k\in Y\setminus\{0\}$, the class of Gerstewitz functionals
$\{\varphi_{A,k} : A\subseteq Y\}$ coincides with the class of $k$-translative functions on $Y$ and with the class of functions $\varphi:Y\to \overline{{\mathbb{R}}}$ having uniform sublevel sets which fulfill the condition
\begin{equation*}
\operatorname*{sublev}\nolimits_{\varphi }(t)=\operatorname*{sublev}\nolimits_{\varphi }(0)+tk\mbox{ for all }t\in\mathbb{R}.
\end{equation*}
\end{theorem} 

Hence, each translative functional can be constructed as a Gerstewitz functional. Moreover, the class of sets $A$ and vectors $k$ which have to be used in order to construct all Gerstewitz functionals can be restricted. Theorem \ref{t-sublevequ}(b) implies:

\begin{corollary}\label{cor-all}
The class of translative functionals 
$$\{\varphi_{A,k} \colon A\subseteq Y, k\in Y\setminus\{ 0\} \}$$ 
coincides with 
$$\{ \varphi_{A,k}\colon A\subseteq Y \mbox{ is } k\mbox{-directionally closed}, k\in -0^+A\setminus\{0\} \}.$$
\end{corollary}

Note that $A-\mathbb{R}_{+}k=A$ if and only if $k\in -0^+A$. Hence, we get from Theorem \ref{t-sublevequ}:

\begin{corollary}\label{c-sublev}
Assume $A\subseteq Y$ and $k\in -0^+A\setminus\{ 0\}$. Then
$$\operatorname*{sublev}\nolimits_{\varphi_{A,k}}(t)=\operatorname*{cl_k}A+tk \mbox{ for all } t \in\mathbb{R}.$$
\end{corollary}

Let us now investigate semicontinuity and continuity of translative functions.

Semicontinuity can be completely characterized by the sublevel sets of the function on the one hand and by the epigraph of the function on the other hand \cite{Mor67}.

\begin{lemma}\label{l-semicon}
Let $Y$ be a topological space, $\varphi: Y\to \overline{\mathbb{R}}$. 
Then the following statements are equivalent:
\begin{itemize}
\item[\rm (a)] $\varphi$ is lower semicontinuous.
\item[\rm (b)] The sublevel sets $\operatorname*{sublev}_{\varphi }(t)$ 
are closed for all $t\in \mathbb{R}$. 
\item[\rm (c)] $ \operatorname*{epi}\varphi$ is closed in $Y\times \mathbb{R}$. 
\end{itemize}
\end{lemma}

Theorem \ref{t-sublevequ} yields because of Lemma \ref{l-semicon}:
\begin{proposition}
Assume that $Y$ is a topological vector space, $A\subseteq Y$ and $k\in Y\setminus\{ 0\}$.
Then $\varphi _{A,k}$ is lower semicontinuous if and only if $\operatorname*{cl_k}(A-\mathbb{R}_{+}k)$ is a closed set.
\end{proposition}

Properties of lower semicontinuous Gerstewitz functionals have been studied in \cite{Wei17a}. Theorem 2.9 from \cite{Wei17a} contains the statement of the next lemma for proper subsets $A$ of $Y$. If $A$ is not a proper subset of $Y$, the statement is obvious.

\begin{lemma}\label{l-cont}
Assume that $Y$ is a topological vector space, $A\subseteq Y$ and $k\in Y\setminus\{0\}$ with $\operatorname*{cl}A-\mathbb{R}_{>}k\subseteq A$. 
$\varphi  _{A,k} $ is continuous if and only if $\operatorname*{cl}A-\mathbb{R}_{>}k\subseteq \operatorname*{int}A$ holds.
\end{lemma}

\begin{theorem}\label{t-semicon-alle}
Assume that $Y$ is a topological vector space, $A\subseteq Y$ and $k\in -0^+A\setminus\{ 0\}$.
\begin{itemize}
\item[\rm (1)]
The following conditions are equivalent to each other:
\begin{itemize}
\item[\rm (a)] $\varphi _{A,k}$ is lower semicontinuous.
\item[\rm (b)] $\operatorname*{cl_k}A=\operatorname*{cl}A$.
\item[\rm (c)] $\operatorname*{sublev}\nolimits_{\varphi_{A,k}}(t)=\operatorname*{cl}A+tk$
holds for all $t \in\mathbb{R}$.
\item[\rm (d)]  $\operatorname*{cl}A-\mathbb{R}_{>}k\subseteq A$.
\end{itemize}
\item[\rm (2)] $\varphi  _{A,k} $ is continuous if and only if $\operatorname*{cl}A-\mathbb{R}_{>}k\subseteq \operatorname*{int}A$.
\end{itemize}
\end{theorem}
\begin{proof}
\begin{itemize}
\item[] 
\item[\rm (1)] 
Because of Corollary \ref{c-sublev} and Lemma \ref{l-semicon}, $\varphi _{A,k}$ is lower semicontinuous if and only if $\operatorname*{cl_k}A$ is a closed set. This is equivalent to  $\operatorname*{cl_k}A=\operatorname*{cl}A$ by Proposition \ref{p-kcl-cl}(d).\\
(b) is equivalent to (c) because of  Corollary \ref{c-sublev}.\\
(b) is equivalent to (d) by  Proposition \ref{p-kcl-cl}(e).
\item[\rm (2)] Assume first that $\varphi  _{A,k} $ is continuous. Then it is lower semicontinous, and $\operatorname*{cl}A-\mathbb{R}_{>}k\subseteq A$ by (1). Lemma \ref{l-cont} implies $\operatorname*{cl}A-\mathbb{R}_{>}k\subseteq \operatorname*{int}A$.\\
Assume now $\operatorname*{cl}A-\mathbb{R}_{>}k\subseteq \operatorname*{int}A$. Obviously, $\operatorname*{cl}A-\mathbb{R}_{>}k\subseteq A$. Lemma \ref{l-cont} implies that $\varphi  _{A,k} $ is continuous.
\end{itemize}
\end{proof}

We can restrict the class of sets $A$ and vectors $k$ which have to be used in order to construct all semicontinuous and continuous translative functionals.

\begin{theorem}
Assume $Y$ to be a topological vector space.
\begin{itemize}
\item[\rm (1)] 
The class of lower semicontinuous translative functionals 
\begin{equation}\label{lsc-class1}
\{ \varphi_{A,k} \colon A\subseteq Y, k\in Y\setminus\{  0\} , \varphi_{A,k} \mbox{ lower semicontinuous} \} \end{equation}
coincides with 
\begin{equation}\label{lsc-class2}
\{ \varphi_{A,k}\colon A \subseteq Y, k\in Y\setminus\{ 0\} , \operatorname*{cl}A-\mathbb{R}_{>}k\subseteq A\}
\end{equation}
and with
\begin{equation}\label{lsc-class3}
\{ \varphi_{A,k}\colon A \subseteq Y \mbox{ closed}, k\in -0^+A\setminus\{ 0\} \}.
\end{equation}
\item[\rm (2)] The class of continuous translative functionals 
\begin{equation}\label{cont-class1}
\{\varphi_{A,k} : A\subseteq Y, k\in Y\setminus\{ 0\} , \varphi_{A,k} \mbox{ continuous} \}
\end{equation}
coincides with 
\begin{equation}\label{cont-class2}
\{ \varphi_{A,k}\colon A\subseteq Y, k\in Y\setminus\{ 0\} , \operatorname*{cl}A-\mathbb{R}_{>}k\subseteq \operatorname*{int}A\}
\end{equation}
and with
\begin{equation}\label{cont-class3}
\{ \varphi_{A,k}\colon A \subseteq Y \mbox{ closed}, k\in Y\setminus\{ 0\}, A-\mathbb{R}_{>}k\subseteq \operatorname*{int}A \}.
\end{equation}
\end{itemize}
\end{theorem}
\begin{proof}
\begin{itemize}
\item[] 
\item[\rm (1)] 
\begin{itemize}
\item[\rm (a)]Take first any $A\subseteq Y$ and $k\in Y\setminus\{  0\}$ for which $\varphi :=\varphi_{A,k}$ is lower semicontinuous. Because of Corollary \ref{cor-all}, there exist $\tilde{A}\subseteq Y$ and $\tilde{k}\in -0^+\tilde{A}\setminus\{  0\}$ with $\varphi =\varphi_{\tilde{A},\tilde{k}}$. 
Theorem \ref{t-semicon-alle} implies $\operatorname*{cl}\tilde{A}-\mathbb{R}_{>}\tilde{k}\subseteq \tilde{A}$. Hence, set (\ref{lsc-class1}) is contained in set (\ref{lsc-class2}).
\item[\rm (b)] Take now any $A \subseteq Y$ and $k\in Y\setminus\{ 0\}$ with $\operatorname*{cl}A-\mathbb{R}_{>}k\subseteq A$.
Consider $\varphi :=\varphi_{A,k}$. $\operatorname*{cl}A-\mathbb{R}_{>}k\subseteq A$ implies $k\in -0^+A\setminus\{  0\}$ and $k\in -0^+(\operatorname*{cl}A)\setminus\{  0\}$. Theorem \ref{t-semicon-alle} implies  $\operatorname*{sublev}\nolimits_{\varphi }(t)=\operatorname*{cl}A+tk $
for all $t \in\mathbb{R}$. Proposition \ref{prop-vor-theo-and} yields $\varphi =\varphi_{\operatorname*{cl}A,k}$.
Hence, set (\ref{lsc-class2}) is contained in set (\ref{lsc-class3}).
\item[\rm (c)] Take any closed set $A \subseteq Y$ and $k\in -0^+A\setminus\{ 0\}$. 
Then $\operatorname*{cl}A-\mathbb{R}_{>}k=A-\mathbb{R}_{>}k\subseteq A$. Theorem \ref{t-semicon-alle} implies that $\varphi_{A,k}$ is lower semicontinuous. Hence, set (\ref{lsc-class3}) is contained in set (\ref{lsc-class1}).
\end{itemize}
\item[\rm (2)] 
\begin{itemize}
\item[\rm (a)]Take first any $A\subseteq Y$ and $k\in Y\setminus\{  0\}$ for which $\varphi :=\varphi_{A,k}$ is continuous. Because of Corollary \ref{cor-all}, there exist $\tilde{A}\subseteq Y$ and $\tilde{k}\in -0^+\tilde{A}\setminus\{  0\}$ with $\varphi =\varphi_{\tilde{A},\tilde{k}}$. 
Theorem \ref{t-semicon-alle} implies $\operatorname*{cl}\tilde{A}-\mathbb{R}_{>}\tilde{k}\subseteq \operatorname*{int}\tilde{A}$.
Hence, set (\ref{cont-class1}) is contained in set (\ref{cont-class2}).
\item[\rm (b)] Take now any $A \subseteq Y$ and $k\in Y\setminus\{ 0\}$ with $\operatorname*{cl}A-\mathbb{R}_{>}k\subseteq \operatorname*{int}A$.
Consider $\varphi :=\varphi_{A,k}$. $\operatorname*{cl}A-\mathbb{R}_{>}k\subseteq \operatorname*{int}A$ implies $k\in -0^+A\setminus\{  0\}$. Theorem \ref{t-semicon-alle} yields $\operatorname*{sublev}\nolimits_{\varphi_{A,k}}(t)=\operatorname*{cl}A+tk $
for all $t \in\mathbb{R}$. Hence, $\varphi =\varphi_{\operatorname*{cl}A,k}$ by Proposition \ref{prop-vor-theo-and}.
Because of $\operatorname*{cl}A-\mathbb{R}_{>}k\subseteq \operatorname*{int}A\subseteq \operatorname*{int}(\operatorname*{cl}A)$, set (\ref{cont-class2}) is contained in set (\ref{cont-class3}).
\item[\rm (c)] Take any closed set $A \subseteq Y$ and $k\in Y\setminus\{ 0\}$ with $A-\mathbb{R}_{>}k\subseteq \operatorname*{int}A$. $\varphi_{A,k}$ is continuous by Theorem \ref{t-semicon-alle}. 
Hence, set (\ref{cont-class3}) is contained in set (\ref{cont-class1}).
\end{itemize}
\end{itemize}
\end{proof}

\begin{remark}
For topological vector spaces $Y$ and $k\in Y\setminus\{  0\}$, the equivalence between 
$\{ \varphi \colon Y\to\overline{\mathbb{R}}\colon \varphi \; k\mbox{-translative and lower semicontinuous} \}$
and $\{ \varphi_{A,k}\colon A \subseteq Y \mbox{ closed}, \, k\in -0^+A\setminus\{  0\}\}$
was proved in \cite[Corollary 8]{Ham06}, and the equivalence between 
$\{ \varphi \colon Y\to\overline{\mathbb{R}}\colon \varphi \; k\mbox{-translative and continuous} \}$
and $\{ \varphi_{A,k}\colon A \subseteq Y \mbox{ closed}, \, A-\mathbb{R}_{>}k\subseteq \operatorname*{int}A \}$
can be deduced from \cite[Corollary 9]{Ham06}.
\end{remark}

Let us mention some properties of Gerstewitz functionals. 

\begin{definition}\label{d-cx_ua}
A functional $\varphi :Y\to \overline{\mathbb{R}}$ is said to be
\begin{itemize}
\item[\rm (a)] convex if  $\operatorname*{epi}\varphi$ is convex, 
\item[\rm (b)] positively homogeneous if  $\operatorname*{epi}\varphi$ is a nonempty cone,
\item[\rm (c)] subadditive if  $\operatorname*{epi}\varphi +\operatorname*{epi}\varphi \subseteq \operatorname*{epi}\varphi$,
\item[\rm (d)] sublinear if $\operatorname*{epi}\varphi$ is a nonempty convex cone.
\end{itemize}
\end{definition}

If $\varphi$ is finite-valued, the above properties are equivalent to those for real-valued functions \cite{Wei18c}.

We get from \cite[Theorem 2.3]{Wei17b}:

\begin{proposition}\label{p-cx_ua}
Assume that $A$ is a proper, $k$-directionally closed subset of $\,Y$ with $k\in -0^+A\setminus\{0\}$. 
\begin{itemize}
\item[\rm (a)] $\varphi  _{A,k} $ is convex if and only if $A$ is convex.
\item[\rm (b)] $\varphi  _{A,k} $ is positively homogeneous if and only if $A$ is a cone.
\item[\rm (c)] $\varphi  _{A,k} $ is subadditive if and only if $A+A\subseteq A$.
\item[\rm (d)] $\varphi  _{A,k} $ is sublinear if and only if $A$ is a convex cone.
\end{itemize}
\end{proposition}

The following statement was proved in \cite[Theorem 2.2]{Wei17b}.

\begin{proposition}\label{p-finval}
Assume that $A$ is a proper subset of $\,Y$ and $k\in -\operatorname*{core}0^+A$. 
Then $\varphi  _{A,k} $ is finite-valued.
\end{proposition}

We illustrate the statements of this section by an example.

\begin{example}\label{ex-tildeA}
In $Y=\mathbb{R}^2$, choose $A:=\{ (y_1,y_2)^T\in Y\colon 0< y_1\leq 2, 0\leq y_2\leq 2\}$ and $k:=(-1,-1)^T$.
Then
$A-\mathbb{R}_{+}k=\{ (y_1,y_2)^T\in Y\colon 0< y_1< 2, 0\leq y_2 < y_1+2\}\cup
\{ (y_1,y_2)^T\in Y\colon y_1\geq 2, y_1-2\leq y_2 < y_1+2\}.$
$\tilde{A}:=\operatorname*{cl}_k(A-\mathbb{R}_{+}k)=\{ (y_1,y_2)^T\in Y\colon 0\leq y_1< 2, 0\leq y_2 < y_1+2\}\cup \{ (y_1,y_2)^T\in Y\colon y_1\geq 2, y_1-2\leq y_2 < y_1+2\}$, and
$\operatorname*{dom}\varphi_{\tilde{A},k}=\{ (y_1,y_2)^T\in Y\colon y_1-2\leq y_2 < y_1+2\}.$ By Theorem \ref{t-sublevequ}, $\varphi_{A,k}=\varphi_{\tilde{A},k}$, $k\in -0^+\tilde{A}\setminus\{ 0\}$ and
$$\operatorname*{sublev}\nolimits_{\varphi_{A,k}}(t)=\tilde{A}+tk \mbox{ for all } t \in\mathbb{R}.$$
We have
\[ \varphi_{A,k}(y)=\left\{
\begin{array}{r@{\quad\mbox{ if }}l}
t & y\in (\operatorname*{bd}\mathbb{R}^{2}_{+}+tk)\cap \operatorname*{dom}\varphi_{A,k},\\
+\infty & y\in Y\setminus \operatorname*{dom}\varphi_{A,k}.
\end{array}
\right.
\]
$\varphi_{A,k}$ is convex by Proposition \ref{p-cx_ua}. Since $\operatorname*{cl}\tilde{A}-\mathbb{R}_{>}k\not\subseteq \tilde{A}$, Theorem \ref{t-semicon-alle} implies that $\varphi_{A,k}$ is not lower semicontinuous.\\ 
Note that the restriction of $\varphi_{A,k}$  to its effective domain is a continuous function. For $\bar{A}:=A\cup \{ (0,2)^T\}$, we get a convex, lower semicontinuous functional $\varphi_{\bar{A},k}$, which coincides with $\varphi_{A,k}$ on $\operatorname*{dom}\varphi_{A,k}$. Since $\operatorname*{cl}\bar{A}-\mathbb{R}_{>}k\not\subseteq \operatorname*{int}\bar{A}$, Theorem \ref{t-semicon-alle} implies that $\varphi_{\bar{A},k}$ is not continuous. But the restriction of $\varphi_{\bar{A},k}$ to its effective domain is a continuous function.
\end{example}

\section{Translative functions with sublevel sets  given by linear inequalities}\label{sec-poly-fktal}

We will now derive some details for functionals $\varphi_{A,k}$ with $A$ being given by linear inequalities. This includes the case that $A$ is a polyhedral cone.

In this section, we start with sets $A$ which are given by linear inequality constraints in a topological vector space $Y$. These results will be applied to polyhedral sets in the Euclidean space. We will work with sets $A$ fulfilling the following assumption:

\vspace{0.2cm}
\begin{tabular}{ll}
(HT$_{A}$): & $\quad A$ is a proper subset of the topological vector space $Y$\\
& \quad which is given by
$A=\{y\in Y\colon g(y)\leq b_g \mbox{ for all } g\in \Gamma \}$\\
& $\quad\mbox{with }\Gamma\subset Y^{*}\setminus\{ 0\} \mbox{ and } b_g\in\mathbb{R} \mbox{ for each } g\in \Gamma $\\
& $\quad$ such that $0^+A\not=\{ 0\} ,$
\end{tabular}

\vspace{0.2cm}
where $Y^{*}$ denotes the space of all continuous linear mappings from $Y$ into $\mathbb{R}$.

Consider any set $A$ which fulfills (HT$_{A}$). It is easy to check that $A$ is an unbounded, closed, convex  subset of $Y$ and that 
$$0^+A=\{y\in Y\colon g(y)\leq 0 \mbox{ for all } g\in \Gamma\}.$$ 
$k\in -0^+A$ is equivalent to the condition that $g(k)\geq 0$ holds for all $g\in \Gamma$. 

\begin{proposition}\label{Fact4-HT}
Assume {\rm (HT}$_{A}${\rm )} and $k\in Y$ such that $g(k)\geq 0$ holds for all $g\in \Gamma$ and that there exists some $g\in \Gamma$ with $g(k)>0$.
\begin{itemize}
\item[\rm (a)] $\varphi _{A,k}$ is convex, proper and lower semicontinuous on $\operatorname*{dom}\varphi _{A,k}$.
\item[\rm (b)] If $\operatorname*{int}A\not=\emptyset$, then $\varphi _{A,k}$ is continuous on $\operatorname*{int}\operatorname*{dom}\varphi _{A,k}$.
\item[\rm (c)] If $\,Y$ is a Banach space, then $\varphi_{A,k}$ is locally Lipschitz on $\operatorname*{int}\operatorname*{dom}\varphi_{A,k}$.
\item[\rm (d)] $\varphi _{A,k}$ is subadditive if and only if $b_g\leq 0$ for all $g\in \Gamma$.
\item[\rm (e)] $\varphi _{A,k}$ is sublinear if and only if $b_g= 0$ for all $g\in \Gamma$.
\end{itemize}
\end{proposition}

\begin{proof}
(a)-(c) follow from \cite[Prop. 4.5]{Wei17a}, where the statements were proved for proper closed convex sets $A$ with $k\in -0^+A\setminus 0^+A$ and $A-\mathbb{R}_>k\subseteq A$.\\
According to Proposition \ref{p-cx_ua}, $\varphi _{A,k}$ is subadditive if and only $A+A\subseteq A$. This proposition also states that $\varphi _{A,k}$ is sublinear if and only if $A$ is a convex cone. 
\end{proof}

The formula for the functional $\varphi _{A,k}$ is described in the following proposition.

\begin{proposition}\label{p-phi-inequ}
Assume {\rm (HT}$_{A}${\rm )}.\\
Take any $k\in Y\setminus\{0\}$ with $g(k)\geq 0$ for all $g\in \Gamma$.
\begin{itemize}
\item[\rm (a)] If $\,g(k)=0$ for all $g\in \Gamma $, we have $\operatorname*{dom}\varphi _{A,k}=A$ and
$$\varphi _{A,k}(y)=-\infty \mbox{ for each } y\in A.$$
\item[\rm (b)] If there exist some $g_1\in \Gamma $ such that $\,g_1(k) >0$, we get\\
$\operatorname*{dom}\varphi _{A,k}=\{y\in Y\colon g(y)\leq b_g\mbox{ for all } g\in {\Gamma }^{{\rm act}} \mbox{ and } \sup_{g\notin {\Gamma }^{{\rm act}}}\frac{g(y)- b_g}{g(k)}\in\mathbb{R}\} $\\
and 
$$\varphi _{A,k}(y)=\sup_{g\notin {\Gamma }^{{\rm act}}}\frac{g(y)- b_g}{g(k)} \mbox{ for all }  y\in\operatorname*{dom}\varphi _{A,k},$$ 
where ${\Gamma }^{{\rm act}}:=\{g\in \Gamma \colon g(k)=0\}$.
\item[\rm (c)] If $\Gamma$ consists of a finite number of functions and $g(k)> 0 \mbox{ for all } g\in \Gamma$, then $\varphi _{A,k}$ is finite-valued with
$$\varphi _{A,k}(y)=\max_{g\in \Gamma }\frac{g(y)- b_g}{g(k)}\mbox{ for all } y\in Y.$$
\end{itemize}
\end{proposition}

\begin{proof}
$\forall y\in Y, t\in\mathbb{R}:\,
y\in A+tk\Leftrightarrow y-tk\in A\Leftrightarrow (\forall\, g\in \Gamma:\,g(y-tk)\leq b_g) \Leftrightarrow (\forall\, g\in \Gamma:\,g(y)-tg(k)\leq b_g).$\\
In the case {\rm (a)}, this yields: $\forall y\in Y, t\in\mathbb{R}:\,
y\in A+tk\Leftrightarrow y\in A$, which results in the assertion for this case since $\operatorname*{dom}\varphi _{A,k}=A+\mathbb{R}k$ and by the definition of $\varphi _{A,k}$.\\
If there exist some $g_1\in \Gamma $ such that $\,g_1(k) > 0$, then: $\forall y\in Y, t\in\mathbb{R}:\,
y\in A+tk\Leftrightarrow (g(y)\leq b_g\mbox{ for all } g\in {\Gamma }^{{\rm act}} \mbox{ and } t\geq \frac{g(y)- b_g}{g(k)}\mbox{ for all } g\notin {\Gamma }^{{\rm act}}).$ This yields {\rm (b)}.\\
{\rm (c)} follows from (b). 
\end{proof}

We will now investigate in which way shifts in the constraints of $A$ influence the values of $\varphi_{A,k}$ and how perturbations in the function value of $\varphi_{A,k}$ can be interpreted as perturbations in the constraints of $A$.

\begin{proposition}\label{inequ-shift}
Assume {\rm (HT$_{A}$)}.\\
Take any $k\in Y\setminus\{0\}$ with $g(k)\geq 0$ for all $g\in \Gamma$.
\begin{itemize}
\item[\rm (a)] For $\tilde{A}:=\{y\in Y\colon g(y)\leq b_g + \epsilon g(k) \mbox{ for all } g\in \Gamma \}$
with $\epsilon\in\mathbb{R}$, we get 
$$\varphi_{\tilde{A},k}(y)=\varphi_{A,k}(y) -\epsilon \mbox{ for all } y\in Y.$$
\item[\rm (b)] For $\tilde{A}:=\{y\in Y\colon g(y)\leq b_g + g(y^0) \mbox{ for all } g\in \Gamma \}$ with $y^0\in Y$, we have 
$$\varphi_{\tilde{A},k}(y)=\varphi_{A,k}(y-y^0) \mbox{ for all } y\in Y.$$
\end{itemize}
\end{proposition}

\begin{proof}
\begin{itemize}
\item[]
\item[\rm (a)] The definition of $\varphi_{A,k}$ implies $\varphi_{A+\epsilon k,k}(y)=\varphi_{A,k}(y) -\epsilon$ for all $y\in Y$.
\begin{eqnarray*}
A+\epsilon k & = & \{y+\epsilon k\in Y\colon g(y)\leq b_g \mbox{ for all } g\in \Gamma\}\\
 & = & \{y\in Y\colon g(y- \epsilon k)\leq b_g \mbox{ for all } g\in \Gamma\} \\
 & = & \{y\in Y\colon g(y)- \epsilon g(k)\leq b_g \mbox{ for all } g\in \Gamma\} \\
 & = & \{y\in Y\colon g(y)\leq b_g+ \epsilon g(k) \mbox{ for all } g\in \Gamma\} 
\end{eqnarray*} 
\item[\rm (b)] The definition of $\varphi_{A,k}$ yields $\varphi_{A+y^0,k}(y)=\varphi_{A,k}(y-y^0)$ for all $y\in Y$.
\begin{eqnarray*}
A+y^0 & = & \{y+y^0\in Y\colon g(y)\leq b_g \mbox{ for all } g\in \Gamma\}\\
 & = & \{y\in Y\colon g(y- y^0)\leq b_g \mbox{ for all } g\in \Gamma\} \\
 & = & \{y\in Y\colon g(y)\leq b_g+ g(y^0)  \mbox{ for all } g\in \Gamma\} 
\end{eqnarray*} 
\end{itemize}
\end{proof}

Let us now consider polyhedral sets $A$ in the Euclidean space which are given by the assumption

\vspace{0.2cm}
\begin{tabular}{ll}
(H$_{A}$): & $\quad A=\{y\in\mathbb{R}^{\ell }\colon Wy\leq b\}$ is an unbounded proper subset of $Y=\mathbb{R}^{\ell }$\\
& \quad with  $b\in\mathbb{R}^{r}$ and $W\in\mathbb{R}^{r,\ell }$, where $W=(w^1 \cdots w^r)^T$ with\\
& \quad $w^i\in \mathbb{R}^{\ell }\setminus\{ 0\}$ for all $i\in\{1,\ldots ,r\}$.
\end{tabular}

\vspace{0.2cm}
Consider any set $A$ which fulfills (H$_{A}$). It is easy to check that 
$$0^+A=\{y\in\mathbb{R}^{\ell }\colon Wy\leq 0\}.$$ 
Since $A$ is an unbounded closed convex set, $0^+A\not=\{ 0\}.$ $k\in -0^+A$ is equivalent to $Wk\geq 0$.

Proposition \ref{Fact4-HT} implies:

\begin{proposition}\label{Fact4-poly}
Assume {\rm (H$_{A}$)}. Take any $k\in \mathbb{R}^{\ell }\setminus\{0\}$ with $Wk\geq 0$ and $Wk\not= 0$.
\begin{itemize}
\item[\rm (a)] $\varphi _{A,k}$ is convex, proper and lower semicontinuous on $\operatorname*{dom}\varphi _{A,k}$.
\item[\rm (b)] $\varphi _{A,k}$ is locally Lipschitz and, hence, continuous on $\operatorname*{int}\operatorname*{dom}\varphi_{A,k}$.
\item[\rm (c)] $\varphi _{A,k}$ is subadditive if and only if $b\leq 0$.
\item[\rm (d)] $\varphi _{A,k}$ is sublinear if and only if $b=0$.
\end{itemize}
\end{proposition}

Proposition \ref{p-phi-inequ} yields the description of the functional $\varphi _{A,k}$.

\begin{proposition}\label{p-phi-poly}
Assume {\rm (H$_{A}$)}. Take any $k\in \mathbb{R}^{\ell }\setminus\{0\}$ with $Wk\geq 0$.
\begin{itemize}
\item[\rm (a)] If $\,Wk=0$, we have $\operatorname*{dom}\varphi _{A,k}=A$ and $\varphi _{A,k}(y)=-\infty$ for each $y\in A$.
\item[\rm (b)] If $\,Wk\not= 0$ and $Wk\not> 0$, we get
$$\operatorname*{dom}\varphi _{A,k}=\{y\in\mathbb{R}^{\ell }\colon (w^i)^Ty\leq b_i\mbox{ for all } i\in I^{{\rm act}}\} \mbox{ and}$$
$$\varphi _{A,k}(y)=\max_{i\notin I^{{\rm act}}}\frac{(w^i)^Ty- b_i}{(w^i)^Tk} \mbox{ for all } y\in\operatorname*{dom}\varphi _{A,k},$$
where $I^{{\rm act}}:=\{i\in\{1,\ldots ,r\} \colon (w^i)^Tk=0\}$.
\item[\rm (c)] In the case $Wk>0$, the function $\varphi _{A,k}$ is finite-valued and
$$\varphi _{A,k}(y)=\max_{i=1,\ldots ,r}\frac{(w^i)^Ty- b_i}{(w^i)^Tk}\quad\mbox{ for all }y\in \mathbb{R}^{\ell }.$$
\end{itemize}
\end{proposition}

If the set $A$ in Proposition \ref{p-phi-poly} is a polyhedral cone, i.e., if $b=0$, then the effective domain  
$\operatorname*{dom}\varphi _{A,k}$ in part (b) is just the Clarke tangent cone of $A$ at $k$ by \cite[p.138]{HirLem93}.

Choose now the identity matrix as matrix $W$.

\begin{corollary}\label{c-phi-id}
Assume $A=b-\mathbb{R}^{\ell }_+$ with $b\in\mathbb{R}^{\ell }$ and $k\in \mathbb{R}^{\ell }_+\setminus\{0\}$.\\
If $k\not> 0$, we get 
$$\operatorname*{dom}\varphi _{A,k}=\{y\in\mathbb{R}^{\ell }\colon y_i\leq b_i\mbox{ for all } i\in I^{{\rm act}}\} \mbox{ and}$$
$$\varphi _{A,k}(y)=\max_{i\notin I^{{\rm act}}}\frac{y_i- b_i}{k_i} \mbox{ for each } y\in\operatorname*{dom}\varphi _{A,k},$$
where $I^{{\rm act}}:=\{i\in\{1,\ldots ,\ell \} \colon k_i=0\}$.\\
In the case $k>0$, the function $\varphi _{A,k}$ is finite-valued and
\begin{equation}\label{f-phi-id}
\varphi _{A,k}(y)=\max_{i=1,\ldots ,\ell }\frac{y_i- b_i}{k_i}\quad\mbox{ for all }y\in \mathbb{R}^{\ell }.
\end{equation}
\end{corollary}

Applying Proposition \ref{p-phi-poly} for $r=1$, we get the following statement for half\-spaces $A$.
\begin{corollary}\label{c-phi-weight}
Assume $A=\{y\in\mathbb{R}^{\ell }\colon w^Ty\leq b\}$ with $w\in\mathbb{R}^{\ell }\setminus\{ 0\}$, $b\in\mathbb{R}$.\\
For $\,k\in \mathbb{R}^{\ell }\setminus\{0\}$ with $w^Tk=0$, we have $\operatorname*{dom}\varphi _{A,k}=A$ and
$$\varphi _{A,k}(y)=-\infty \mbox{ for all } y\in A.$$
For $\,k\in \mathbb{R}^{\ell }\setminus\{0\}$ with $\,w^Tk > 0$, the function $\varphi _{A,k}$ is finite-valued and
\begin{equation}\label{f-phi-weight}
\varphi _{A,k}(y)=\frac{w^Ty- b}{w^Tk}\quad\mbox{ for all }y\in \mathbb{R}^{\ell }.
\end{equation}
\end{corollary}

We get some sensitivity results from Proposition \ref{inequ-shift}:

\begin{proposition}\label{poly-shift}
Assume {\rm (H$_{A}$)}. Take any $k\in \mathbb{R}^{\ell }\setminus\{0\}$ with $Wk\geq 0$.
\begin{itemize}
\item[\rm (a)] For $\tilde{A}:=\{y\in\mathbb{R}^{\ell }\colon Wy\leq b+\epsilon Wk\}$ with $\epsilon\in\mathbb{R}$, we get 
$$\varphi_{\tilde{A},k}(y)=\varphi_{A,k}(y) -\epsilon \mbox{ for all } y\in\mathbb{R}^{\ell }.$$
\item[\rm (b)] For $\tilde{A}:=\{y\in\mathbb{R}^{\ell }\colon Wy\leq b+ Wy^0\}$ with $y^0\in\mathbb{R}^{\ell }$, we have 
$$\varphi_{\tilde{A},k}(y)=\varphi_{A,k}(y-y^0)\mbox{ for all }y\in\mathbb{R}^{\ell }.$$
\end{itemize}
\end{proposition}

\begin{corollary}
Assume that $A=\{y\in\mathbb{R}^{\ell }\colon Wy\leq b\}$ holds for some  $b\in\mathbb{R}^{\ell }$ and 
some regular matrix $W\in\mathbb{R}^{\ell ,\ell }$.
Take any $k\in \mathbb{R}^{\ell }\setminus\{0\}$ with $Wk\geq 0$.
For 
$$\tilde{A}:=\{y\in\mathbb{R}^{\ell }\colon Wy\leq b+ s\}\mbox{ with }s\in\mathbb{R}^{\ell },$$ 
we have $\varphi_{\tilde{A},k}(y)=\varphi_{A,k}(y-W^{-1}s)$ for all $y\in\mathbb{R}^{\ell }$.
\end{corollary}

\begin{remark}
Tammer and Winkler \cite{TamWin03} studied $\varphi _{A,k}$ for polyhedral sets $A$ the construction of which ones depends on the unit ball of some block norm. These sets $A$ fulfill the property $\operatorname*{bd}A-\mathbb{R}_+^{\ell }\subseteq A$. For this special case, Tammer and Winkler proved 
the formula from Proposition \ref{p-phi-poly}(c) for $k\in\operatorname*{int}\mathbb{R}_+^{\ell }$ 
and that $\varphi _{A,k}$ is finite-valued, continuous, and convex. 
The formula from Proposition \ref{p-phi-poly}(c) was also given in \cite[Ex.3.27]{KKTZ2015} for the special case that $A$ is a polyhedral cone with $W\geq 0$. The formula (\ref{f-phi-id}) was shown for $A=b-\mathbb{R}_+^{\ell }$ and $k>0$ in \cite[p.14]{CheHuaYan05}. The special case of the formula (\ref{f-phi-weight}) under the additional assumptions $b=0$ and $w>0$ can be found in \cite[Ex.5]{KoeTam15}.
\end{remark}

\section{Extension of functions to translative functions}\label{s-arbitfunc}

In this section, each extended real-valued function will be shown to be the restriction of some translative functional to a hyperspace. 

\begin{proposition}
Assume 
\begin{equation}\label{hyp-epi}
\tag{{\rm H$_{\rm epi}$}}
f\colon Y\to\overline{\mathbb{R}}, \; A:=\operatorname*{epi}f, \; k:=(0 _{Y},-1)\in Y\times \mathbb{R}.
\end{equation}
Then $\operatorname*{dom}\varphi  _{A,k}=\operatorname*{dom}f \times \mathbb{R}$ and
\begin{equation}\label{eq-vgl}
\varphi _{A,k}((y,s))=f(y)-s \mbox{ for all }(y,s)\in Y\times \mathbb{R}.
\end{equation}
\end{proposition}
\begin{proof}
Take any $(y,s)\in Y\times \mathbb{R}.$
\begin{eqnarray*}
\varphi _{A,k}((y,s)) & = & \operatorname*{inf}\{t\in \mathbb{R}\colon (y,s)\in A +t(0 _{Y},-1)\}\\
& = & \operatorname*{inf}\{t\in \mathbb{R}\colon (y,s+t)\in A\} \\
& = & \operatorname*{inf}\{t\in \mathbb{R}\colon f(y)\leq s+t\} \\
& = & \operatorname*{inf}\{t\in \mathbb{R}\colon f(y)-s\leq t\} \\
& = & f(y)-s 
\end{eqnarray*}
\end{proof}

This implies:

\begin{theorem}
Each extended real-valued function on a linear space is the restriction of some Gerstewitz functional to a hyperspace.\\
In detail, assuming \eqref{hyp-epi}, we get $\operatorname*{dom}f=\{ y\in Y\colon (y,0)\in\operatorname*{dom}\varphi  _{A,k}\}$ and $f(y)=\varphi _{A,k}((y,0))$ for each $y\in Y$.
\end{theorem}

Let us now study interdependencies between the properties of $f$ and $\varphi _{A,k}.$

\begin{lemma}\label{l-epiclsd}
Assume \eqref{hyp-epi}. Then:
\begin{itemize}
\item[\rm (a)] $k\in -0^+A\setminus\{0\}.$
\item[\rm (b)] $A$ is $k$-directionally closed.
\item[\rm (c)] $A$ is a proper subset of $Y\times \mathbb{R}$ if and only if $f$ is not a constant function with the function value $+\infty$ or $-\infty$.
\end{itemize}
\end{lemma}
\begin{proof}
\begin{itemize}
\item[]
\item[\rm (a)] Take any $(y,t)\in A$, $\lambda \in \mathbb{R}_+$.\\
$\Rightarrow f(y)\leq t\leq t+\lambda.$
$\Rightarrow (y,t)-\lambda k=(y,t+\lambda )\in A.$
$\Rightarrow k\in -0^+A\setminus\{0\}.$
\item[\rm (b)] Take any $(y,t)\in Y\times \mathbb{R}$ for which there exists some sequence $(t_n)_{n\in\mathbb{N}}$ of real numbers with $t_n\searrow 0$ and $(y,t)-t_n k\in A$ for all $n\in\mathbb{N}.$
$(y,t+t_n)=(y,t)-t_n k \in A.$
$\Rightarrow f(y)\leq t+t_n$ for all $n\in\mathbb{N}.$ $\Rightarrow f(y)\leq t.$ $\Rightarrow (y,t)\in A.$
Hence, $A$ is $k$-directionally closed.
\item[\rm (c)] $A=\emptyset$ if and only if $\operatorname*{dom}f=\emptyset$.\\
$A=Y\times \mathbb{R}$ if and only if $f(y)= -\infty$ for each $y\in Y$.
\end{itemize}
\end{proof}

Proposition \ref{p-cx_ua} and Definition \ref{d-cx_ua} yield:

\begin{proposition}
Assume \eqref{hyp-epi} and that $f$ is not a constant function  with the function value $+\infty$ or $-\infty$. Then:
\begin{itemize}
\item[\rm (a)] $f$ is convex if and only if $\varphi _{A,k}$ is convex.
\item[\rm (b)] $f$ is positively homogeneous if and only if $\varphi _{A,k}$ is positively homogeneous.
\item[\rm (c)] $f$ is subadditive if and only if $\varphi _{A,k}$ is subadditive.
\item[\rm (d)] $f$ is sublinear if and only if $\varphi _{A,k}$ is sublinear.
\end{itemize}
\end{proposition}

Especially in vector optimization, monotonicity of functionals turns out to be essential for their usability in scalarizing problems.

\begin{definition}\label{d-mon}
Assume $F, B\subseteq Y$.
$\varphi: Y \to \overline{\mathbb{R}}$ is said to be
\begin{itemize}
\item[\rm (a)]
$B$-monotone on $F$
if $y^1,y^2 \in F$ and $y^{2}-y^{1}\in B$ imply $\varphi (y^{1})\le \varphi (y^{2})$,
\item[\rm (b)] strictly $B$-monotone on $F$ 
if $y^1,y^2 \in F$ and $y^{2}-y^{1}\in B\setminus
\{0\}$ imply \linebreak
$\varphi (y^{1})<\varphi (y^{2})$.
\end{itemize}
\end{definition}

\begin{proposition}
Assume \eqref{hyp-epi} and $F, B\subseteq Y$ with $F\not=\emptyset$. Then:
\begin{itemize}
\item[\rm (1)] $f$ is $B$-monotone on $F$ if and only if $\varphi _{A,k}$ is $(B\times (-\mathbb{R}_+))$-monotone on $F\times \mathbb{R}$.
\item[\rm (2)] $\varphi _{A,k}$ is strictly $(B\times (-\mathbb{R}_+))$-monotone on $F\times \mathbb{R}$ if and only if
\begin{itemize}
\item[\rm (a)] $f$ is strictly $B$-monotone on $F$, and
\item[\rm (b)] $f$ is finite-valued on $F$ or $0_Y\notin B$.
\end{itemize}
\end{itemize}
\end{proposition}
\noindent {\it Proof.}
\begin{itemize}
\item[\rm (1)] 
Assume first that $f$ is $B$-monotone on $F$. Take any $(y^1,s_1),(y^2,s_2)\in F\times \mathbb{R}$ with 
$(y^2,s_2)-(y^1,s_1)\in B\times (-\mathbb{R}_+)$. Then $f(y^1)\leq f(y^2)$ since $f$ is $B$-monotone on $F$, and $s_1\geq s_2$. Hence, $\varphi _{A,k}((y^1,s_1))\leq \varphi _{A,k}((y^2,s_2))$ by (\ref{eq-vgl}). Thus, $\varphi _{A,k}$ is $(B\times (-\mathbb{R}_+))$-monotone on $F\times \mathbb{R}$.\\
Assume now that $\varphi _{A,k}$ is $(B\times (-\mathbb{R}_+))$-monotone on $F\times \mathbb{R}$. Take any $y^1, y^2\in F$ with $y^2-y^1\in B$. Then $\varphi _{A,k}((y^1,0))\leq \varphi _{A,k}((y^2,0))$ by the monotonicity. Hence, $f(y^1)\leq f(y^2)$ by (\ref{eq-vgl}). Thus, $f$ is $B$-monotone on $F$.
\item[\rm (2)] 
\begin{itemize}
\item[\rm (i)] 
Assume first that $\varphi _{A,k}$ is strictly $(B\times (-\mathbb{R}_+))$-monotone on $F\times \mathbb{R}$.
Take any $y^1, y^2\in F$ with $y^2-y^1\in B\setminus\{ 0\}$. Then $\varphi _{A,k}((y^1,0))< \varphi _{A,k}((y^2,0))$ by the assumed monotonicity. Hence, $f(y^1)< f(y^2)$ by (\ref{eq-vgl}). Thus, $f$ is strictly $B$-monotone on $F$.\\
Suppose $0\in B$. Take any $y\in F$. $(y,0)-(y,1)\in(B\times (-\mathbb{R}_+))\setminus\{ (0_Y,0)\}$ implies 
$\varphi _{A,k}(y,1)< \varphi _{A,k}(y,0)$. By (\ref{eq-vgl}), $f(y)-1<f(y)-0$. Hence, $f(y)\in\mathbb{R}$. Thus 
$\varphi _{A,k}$ is finite-valued on $F$.
\item[\rm (ii)] Assume now that {\rm (a)} and {\rm (b)} hold. Take any $(y^1,s_1),(y^2,s_2)\in F\times \mathbb{R}$ with 
$(y^2,s_2)-(y^1,s_1)\in (B\times (-\mathbb{R}_+))\setminus\{ (0_Y,0)\}$. 
If $y^1=y^2$, then $s_1> s_2$. Otherwise, $f(y^1) < f(y^2)$ since $f$ is strictly $B$-monotone on $F$, and $s_1 \geq s_2$. If $f$ is finite-valued on $F$, we get $f(y^1)-s_1< f(y^2)-s_2$. If $0_Y\notin B$, then $y^1\not= y^2$,
and we get $f(y^1)-s_1< f(y^2)-s_2$. This implies $\varphi _{A,k}((y^1,s_1))< \varphi _{A,k}((y^2,s_2))$ by (\ref{eq-vgl})
Thus, $\varphi _{A,k}$ is strictly $(B\times (-\mathbb{R}_+))$-monotone on $F\times \mathbb{R}$.
\hfill $\square$
\end{itemize}
\end{itemize}

\begin{proposition}\label{p-contepi}
Assume \eqref{hyp-epi} and that $Y$ is a topological vector space. 
\begin{itemize}
\item[\rm (a)] $f$ is lower semicontinuous if and only if $\varphi _{A,k}$ is lower semicontinuous.
\item[\rm (b)] $f$ is continuous if and only if $\varphi _{A,k}$ is continuous.
\end{itemize}
\end{proposition}
\begin{proof}
\begin{itemize}
\item[] 
\item[\rm (a)] 
Apply Lemma \ref{l-semicon}. $f$ is lower semicontinuous if and only if $A$ is closed. $\varphi _{A,k}$ is lower semicontinuous if and only if $\operatorname*{sublev}_{\varphi _{A,k} }(t)$ is closed for each $t\in\mathbb{R}$. 
By Lemma \ref{l-epiclsd}, $A$ is $k$-directionally closed and $k\in -0^+A\setminus\{ 0\}$. This implies
$\operatorname*{sublev}_{\varphi _{A,k} }(t)=A+tk$ for each $t\in\mathbb{R}$ because of Theorem \ref{t-sublevequ}. Hence, $A$ is closed if and only if $\operatorname*{sublev}_{\varphi _{A,k} }(t)$ is closed for each $t\in\mathbb{R}$.
\item[\rm (b)] results from (\ref{eq-vgl}) by the properties of continuous functions.
\end{itemize}
\end{proof}

This implies a characterization of any continuous function by its epigraph.

\begin{proposition}\label{p-epicont}
Assume that $Y$ is a topological vector space and $f\colon Y\to\overline{\mathbb{R}}$.\\
$f$ is continuous if and only if $\operatorname*{epi}f$ is closed and
\begin{equation}\label{eq-epicont}
\operatorname*{epi}f+\mathbb{R}_{>}\cdot (0 _{Y},1)\subseteq \operatorname*{int}(\operatorname*{epi}f).
\end{equation}
\end{proposition}
\begin{proof}
Define $A:=\operatorname*{epi}f$ and $k:=(0 _{Y},-1)\in Y\times \mathbb{R}$.
By Proposition \ref{p-contepi}, $f$ is continuous if and only if $\varphi _{A,k}$ is continuous. $k\in -0^+A\setminus\{ 0\}$ holds because of Lemma \ref{l-epiclsd}. We get from Theorem \ref{t-semicon-alle} that $\varphi _{A,k}$ is continuous if and only if 
$\operatorname*{cl}A-\mathbb{R}_{>}\cdot k\subseteq \operatorname*{int}A.$  
Hence, $f$ is continuous if and only if 
\begin{equation}\label{eq2}
\operatorname*{cl}A+\mathbb{R}_{>}\cdot (0 _{Y},1)\subseteq \operatorname*{int}A.
\end{equation}
Assume first that $f$ is continuous. Then it is lower semicontinuous. By Lemma \ref{l-semicon}, $A$ is closed, and we get $A+\mathbb{R}_{>}\cdot (0 _{Y},1)=\operatorname*{cl}A+\mathbb{R}_{>}\cdot (0 _{Y},1)\subseteq \operatorname*{int}A.$\\
Assume now that $\operatorname*{epi}f$ is closed and that \eqref{eq-epicont} holds. Then
\eqref{eq2} is fulfilled, and $f$ is continuous.
\end{proof}

\begin{corollary}
Assume that $Y$ is a topological vector space.
Each continuous, positively homogeneous function $f\colon Y\to\overline{\mathbb{R}}$ is finite-valued or the constant function with the function value $-\infty$.
\end{corollary}
\begin{proof}
Since $f$ is positively homogeneous, $A:=\operatorname*{epi}f$ is a nonempty cone. By Proposition \ref{p-epicont},
$A-\mathbb{R}_{>}\cdot k\subseteq \operatorname*{int}A$ for $k:=(0_Y ,-1)$. Hence, $k\in -\operatorname*{int}A=-\operatorname*{int}(0^+A)$. If $A=Y$, then $f$ is the constant function with the function value $-\infty$. Otherwise, $\varphi _{A,k}$ is finite-valued because of Proposition \ref{p-finval}.
Thus, $f$ is finite-valued by \eqref{eq-vgl}.
\end{proof}

\def\cfac#1{\ifmmode\setbox7\hbox{$\accent"5E#1$}\else
  \setbox7\hbox{\accent"5E#1}\penalty 10000\relax\fi\raise 1\ht7
  \hbox{\lower1.15ex\hbox to 1\wd7{\hss\accent"13\hss}}\penalty 10000
  \hskip-1\wd7\penalty 10000\box7}
  \def\cfac#1{\ifmmode\setbox7\hbox{$\accent"5E#1$}\else
  \setbox7\hbox{\accent"5E#1}\penalty 10000\relax\fi\raise 1\ht7
  \hbox{\lower1.15ex\hbox to 1\wd7{\hss\accent"13\hss}}\penalty 10000
  \hskip-1\wd7\penalty 10000\box7}
  \def\cfac#1{\ifmmode\setbox7\hbox{$\accent"5E#1$}\else
  \setbox7\hbox{\accent"5E#1}\penalty 10000\relax\fi\raise 1\ht7
  \hbox{\lower1.15ex\hbox to 1\wd7{\hss\accent"13\hss}}\penalty 10000
  \hskip-1\wd7\penalty 10000\box7}
  \def\cfac#1{\ifmmode\setbox7\hbox{$\accent"5E#1$}\else
  \setbox7\hbox{\accent"5E#1}\penalty 10000\relax\fi\raise 1\ht7
  \hbox{\lower1.15ex\hbox to 1\wd7{\hss\accent"13\hss}}\penalty 10000
  \hskip-1\wd7\penalty 10000\box7}
  \def\cfac#1{\ifmmode\setbox7\hbox{$\accent"5E#1$}\else
  \setbox7\hbox{\accent"5E#1}\penalty 10000\relax\fi\raise 1\ht7
  \hbox{\lower1.15ex\hbox to 1\wd7{\hss\accent"13\hss}}\penalty 10000
  \hskip-1\wd7\penalty 10000\box7}
  \def\cfac#1{\ifmmode\setbox7\hbox{$\accent"5E#1$}\else
  \setbox7\hbox{\accent"5E#1}\penalty 10000\relax\fi\raise 1\ht7
  \hbox{\lower1.15ex\hbox to 1\wd7{\hss\accent"13\hss}}\penalty 10000
  \hskip-1\wd7\penalty 10000\box7}
  \def\cfac#1{\ifmmode\setbox7\hbox{$\accent"5E#1$}\else
  \setbox7\hbox{\accent"5E#1}\penalty 10000\relax\fi\raise 1\ht7
  \hbox{\lower1.15ex\hbox to 1\wd7{\hss\accent"13\hss}}\penalty 10000
  \hskip-1\wd7\penalty 10000\box7}
  \def\cfac#1{\ifmmode\setbox7\hbox{$\accent"5E#1$}\else
  \setbox7\hbox{\accent"5E#1}\penalty 10000\relax\fi\raise 1\ht7
  \hbox{\lower1.15ex\hbox to 1\wd7{\hss\accent"13\hss}}\penalty 10000
  \hskip-1\wd7\penalty 10000\box7}
  \def\cfac#1{\ifmmode\setbox7\hbox{$\accent"5E#1$}\else
  \setbox7\hbox{\accent"5E#1}\penalty 10000\relax\fi\raise 1\ht7
  \hbox{\lower1.15ex\hbox to 1\wd7{\hss\accent"13\hss}}\penalty 10000
  \hskip-1\wd7\penalty 10000\box7} \def\Dbar{\leavevmode\lower.6ex\hbox to
  0pt{\hskip-.23ex \accent"16\hss}D}
  \def\cfac#1{\ifmmode\setbox7\hbox{$\accent"5E#1$}\else
  \setbox7\hbox{\accent"5E#1}\penalty 10000\relax\fi\raise 1\ht7
  \hbox{\lower1.15ex\hbox to 1\wd7{\hss\accent"13\hss}}\penalty 10000
  \hskip-1\wd7\penalty 10000\box7} \def\cprime{$'$}
  \def\Dbar{\leavevmode\lower.6ex\hbox to 0pt{\hskip-.23ex \accent"16\hss}D}
  \def\cfac#1{\ifmmode\setbox7\hbox{$\accent"5E#1$}\else
  \setbox7\hbox{\accent"5E#1}\penalty 10000\relax\fi\raise 1\ht7
  \hbox{\lower1.15ex\hbox to 1\wd7{\hss\accent"13\hss}}\penalty 10000
  \hskip-1\wd7\penalty 10000\box7} \def\cprime{$'$}
  \def\Dbar{\leavevmode\lower.6ex\hbox to 0pt{\hskip-.23ex \accent"16\hss}D}
  \def\cfac#1{\ifmmode\setbox7\hbox{$\accent"5E#1$}\else
  \setbox7\hbox{\accent"5E#1}\penalty 10000\relax\fi\raise 1\ht7
  \hbox{\lower1.15ex\hbox to 1\wd7{\hss\accent"13\hss}}\penalty 10000
  \hskip-1\wd7\penalty 10000\box7}
  \def\udot#1{\ifmmode\oalign{$#1$\crcr\hidewidth.\hidewidth
  }\else\oalign{#1\crcr\hidewidth.\hidewidth}\fi}
  \def\cfac#1{\ifmmode\setbox7\hbox{$\accent"5E#1$}\else
  \setbox7\hbox{\accent"5E#1}\penalty 10000\relax\fi\raise 1\ht7
  \hbox{\lower1.15ex\hbox to 1\wd7{\hss\accent"13\hss}}\penalty 10000
  \hskip-1\wd7\penalty 10000\box7} \def\Dbar{\leavevmode\lower.6ex\hbox to
  0pt{\hskip-.23ex \accent"16\hss}D}
  \def\cfac#1{\ifmmode\setbox7\hbox{$\accent"5E#1$}\else
  \setbox7\hbox{\accent"5E#1}\penalty 10000\relax\fi\raise 1\ht7
  \hbox{\lower1.15ex\hbox to 1\wd7{\hss\accent"13\hss}}\penalty 10000
  \hskip-1\wd7\penalty 10000\box7} \def\Dbar{\leavevmode\lower.6ex\hbox to
  0pt{\hskip-.23ex \accent"16\hss}D}
  \def\cfac#1{\ifmmode\setbox7\hbox{$\accent"5E#1$}\else
  \setbox7\hbox{\accent"5E#1}\penalty 10000\relax\fi\raise 1\ht7
  \hbox{\lower1.15ex\hbox to 1\wd7{\hss\accent"13\hss}}\penalty 10000
  \hskip-1\wd7\penalty 10000\box7} \def\Dbar{\leavevmode\lower.6ex\hbox to
  0pt{\hskip-.23ex \accent"16\hss}D}
  \def\cfac#1{\ifmmode\setbox7\hbox{$\accent"5E#1$}\else
  \setbox7\hbox{\accent"5E#1}\penalty 10000\relax\fi\raise 1\ht7
  \hbox{\lower1.15ex\hbox to 1\wd7{\hss\accent"13\hss}}\penalty 10000
  \hskip-1\wd7\penalty 10000\box7}

\end{document}